\documentclass[11pt]{amsart}

\usepackage{amsmath, amsthm}
\usepackage{amssymb}
\usepackage{amscd}
\usepackage{palatino}
\usepackage{latexsym}
\usepackage{epsfig}
\usepackage{graphics}
\usepackage{amsfonts}
\usepackage{psfrag}
\usepackage{graphicx}

\input xy
\xyoption{all}
\UseComputerModernTips

\numberwithin{equation}{section}
\numberwithin{figure}{section}

\newtheorem{theorem}{Theorem}[section]
\newtheorem{lemma}[theorem]{Lemma}
\newtheorem{proposition}[theorem]{Proposition}

\theoremstyle{definition}
\newtheorem{definition}[theorem]{Definition}

\newtheorem{remark}[theorem]{Remark}

\newcommand{\C}{{\mathbb{C}}}

\newcommand{\Z}{{\mathbb{Z}}}
\newcommand{\Q}{{\mathbb{Q}}}
\newcommand{\R}{{\mathbb{R}}}

\renewcommand{\P}{{\mathbb{P}}}

\renewcommand{\t}{\mathfrak{t}}

\newcommand{\into}{\hookrightarrow}

\renewcommand{\mod}{/\!/}

\DeclareMathOperator{\Stab}{Stab}

\DeclareMathOperator{\Crit}{Crit}

\DeclareMathOperator{\pt}{pt}

\DeclareMathOperator{\lcm}{lcm}

\newcommand{\iNK}{i_{\mathrm{NK}}}
\newcommand{\NKd}{NK^{\diamond}_{T}}

\newcommand{\okf}{\mathsf{K}_{\mathrm{orb}}}

\newcommand{\ec}{\overline{e}} 
\newcommand{\kt}{\odot} 
\newcommand{\Mt}{\widetilde{M}}

\newcommand{\Kernel}{{\mathcal K}}

\begin{document}

\title{The full orbifold $K$-theory of abelian symplectic quotients}

\author{Rebecca Goldin}
\address{Mathematical Sciences MS 3F2, George Mason University, 4400 University Drive, Fairfax, Virginia 22030,
USA} \email{rgoldin@math.gmu.edu} \urladdr{\tt{http://math.gmu.edu/~rgoldin/}}

\author{Megumi Harada}
\address{Department of Mathematics and
Statistics, McMaster University, 1280 Main Street West, Hamilton,
Ontario L8S4K1, Canada}
\email{Megumi.Harada@math.mcmaster.ca}
\urladdr{\tt{http://www.math.mcmaster.ca/{}Megumi.Harada/}}

\author{Tara S. Holm}
\address{Department of Mathematics, 310 Malott Hall, Cornell
  University, Ithaca, New York 14853-4201, USA}
\email{tsh@math.cornell.edu}
\urladdr{\tt{http://www.math.cornell.edu/~tsh/}}

\author{Takashi Kimura}
\address{Department of Mathematics and Statistics, Boston University,
  111 Cummington Street, Boston, Massachusetts 02215, USA}
\email{kimura@math.bu.edu}
\urladdr{\tt{http://math.bu.edu/people/kimura/}}

\keywords{full orbifold $K$-theory, inertial $K$-theory, Hamiltonian $T$-space, symplectic quotient}
\subjclass[2000]{Primary: 19L47; Secondary: 53D20}

\date{\today}


\begin{abstract}
In their 2007 paper, Jarvis, Kaufmann, and Kimura defined the {\bf
  full orbifold $K$-theory} of an orbifold ${\mathfrak X}$, analogous
  to the Chen-Ruan orbifold cohomology of ${\mathfrak X}$ in that it
  uses the {\bf obstruction bundle} as a quantum correction to the
  multiplicative structure. We give an explicit algorithm for the
  computation of this orbifold invariant in the case when ${\mathfrak
  X}$ arises as an abelian symplectic quotient. Our methods are integral
  $K$-theoretic analogues of those used in the orbifold cohomology
  case by Goldin, Holm, and Knutson in 2005. We rely on the
  $K$-theoretic Kirwan surjectivity methods developed by Harada and
  Landweber. 
As a worked class of examples, we compute the full orbifold $K$-theory of weighted projective
spaces that occur as a symplectic quotient of a complex affine space by a circle.
  Our computations hold over the integers, and in the
  particular case of
these weighted projective spaces, we show that the
  associated invariant is torsion-free.

\end{abstract}

\maketitle

\section{Introduction}\label{sec:intro}

Orbifolds and their invariants, including homotopy groups, cohomology rings, and $K$-theory rings, are an active area of current research.  Much recent work concerns stringy versions of these invariants (which take into account the so-called twisted sectors), motivated by the seminal work of Dixon, Harvey, Vafa, and Witten \cite{DHVW85}.  Examples of such invariants are the {\bf orbifold cohomology} of Chen and Ruan \cite{CR04} and the {\bf full orbifold $K$-theory} introduced by Jarvis, Kaufmann, and Kimura \cite{JarKauKim07} and further developed by Becerra and Uribe \cite{BecUri07}.

The main result of this manuscript is the complete description of the full orbifold
$K$-theory of
abelian symplectic quotients, using techniques from equivariant symplectic
geometry.  The examples to which these methods apply include many
orbifold toric varieties
(or smooth toric Deligne-Mumford stacks) as discussed in e.g.\ \cite{BCS05, BorHor06, GHK05, Iwa07,
JiaTse07}, 
including weighted projective spaces, a topic of active
current research \cite{BahFraRay09, BoiManPer06, CoaCorLeeTse06, GueSak08, Hol07a, Tym08}. Another
class of examples are the orbifold weight varieties of \cite{Knu96}
and \cite[Section 8]{GHK05}.

We introduce a new ring, called {\bf inertial $K$-theory},
associated
to a $T$-action on a manifold $M$, where $T$ is a compact abelian Lie
group and $M$ is a stably 
complex manifold. This ring
generalizes the
{\bf stringy $K$-theory} 
defined by Becerra and Uribe \cite[Definition 2.1]{BecUri07}, which
applies to a 
{locally free} $T$-action on $M$.%
\footnote{What the authors in \cite{BecUri07} 
call 'stringy $K$-theory' is analogous to the full
 orbifold $K$-theory of \cite{JarKauKim07} in the case that the
 orbifold is a global quotient by an abelian Lie group.} In contrast, 
the definition of inertial
 $K$-theory does not require that $T$ act locally freely. An important
 special case is when $X$ is a Hamiltonian $T$-space. In this setting, the restriction map to 
the $T$-fixed
set allows us to simplify the product for the purposes of computation. 
We then use an analogue of the {\bf Kirwan surjectivity theorem} from
equivariant symplectic geometry to prove that the inertial $K$-theory
surjects onto (an integral lift of) the full orbifold $K$-theory of
\cite{JarKauKim07}.

We take a moment to discuss other versions of $K$-theory for orbifolds
discussed in the literature. 
In \cite{JarKauKim07}, the authors also introduce the {\bf stringy $K$-theory} 
$\mathcal{K}(X,G)$ associated to a smooth projective
variety $X$ with an action by a finite group $G$.
\footnote{The stringy $K$-theory 
of \cite{JarKauKim07} is defined only
  when $G$ is finite, and differs from the stringy $K$-theory 
  of Becerra and Uribe \cite{BecUri07}.}
In this case, the
$G$-invariant part of $\mathcal{K}(X,G)$, which in \cite{JarKauKim07}
is called the {\bf small orbifold $K$-theory}, is isomorphic as a
vector space, but not as a ring, to the orbifold
$K$-theory of the global quotient ${\mathfrak X} = [X/G]$ as defined
by Adem and Ruan \cite{AdeRua03}.  In the setting of a global quotient
by a finite group, the full orbifold $K$-theory of \cite{JarKauKim07} contains the small
orbifold $K$-theory $\mathcal{K}(X,G)^G$ as a subring. However, the
full orbifold $K$-theory is far more general; in particular, it may be
defined for stacks which are not global quotients.

Our definition of inertial $K$-theory $\NKd(M)$ follows
ideas 
introduced by Goldin, Holm, and Knutson \cite{GHK05} in the setting of
cohomology.
The ring $\NKd(M)$ is well defined for any stably complex $T$-space $M$.
In the case when the $T$-action is locally free,
$\NKd(M)$ is the stringy $K$-theory of Becerra and Uribe and
in particular is isomorphic to  (an integral lift of)
the full orbifold $K$-theory of
the associated orbifold \cite[Section 2]{BecUri07}.

When the
$T$-action is not locally free, then as far as we are aware, the
inertial $K$-theory ring $\NKd(M)$ is a new ring which has not
appeared previously in the literature and 
does not correspond to the $K$-theory of a stack. The point of
introducing 
$\NKd(M)$ is that we can use it along with integral $K$-theoretic analogues \cite{HarLan07,
HL-kernel} of Kirwan surjectivity to compute 
the full orbifold $K$-theory of orbifolds
${\mathfrak X}$ arising as abelian symplectic quotients. For example, orbifold toric varieties (as
studied e.g. in \cite{LT97}) arise in this manner via the Delzant
construction.

Our computations depend on an inertial-$K$-theory analogue of standard
localization results in equivariant symplectic geometry. 
Specifically, when $M$ is a Hamiltonian $T$-space, the stringy product on $\NKd(M)$ may be reformulated using the $T$-fixed point sets and their normal bundles, simplifying computations.
This is called the $\star$-product, and mimics a
similar product in orbifold cohomology as in \cite{GHK05}. 
Our main theorem, proven in Section~\ref{sec:surjectivity}, is the following.

\begin{theorem}\label{th:main}
Let $T$ be a compact connected abelian Lie group, and $(M, \omega, \Phi)$ a Hamiltonian
$T$-space with proper moment map \(\Phi: M \to \t^*.\) Assume that
\(\alpha \in \t^*\) is a regular value of $\Phi$, so that $T$ acts locally freely
on $\Phi^{-1}(\alpha)$.
Then the inclusion \(\iota:
\Phi^{-1}(\alpha) \into M,\) induces a {\bf ring homomorphism} in
inertial $K$-theory:
\begin{equation}\label{eq:kappa-in-intro}
\xymatrix{
\kappa: \NKd(M) \ar[r]^{\iota^*} & \NKd(\Phi^{-1}(\alpha))
\ar[r]^<<<<<{\cong} &
\okf([\Phi^{-1}(\alpha)/T]) =:  \okf([M \mod_{\! \alpha} T])
}
\end{equation}
from the inertial $K$-theory $\NKd(M)$ of $M$ onto the integral
full orbifold $K$-theory $\okf([M\mod_{\! \alpha} T])$
of the quotient orbifold $[M \mod_{\! \alpha} T] := [\Phi^{-1}(\alpha)/T]$.
Furthermore, this map
is {\bf surjective}.
\end{theorem}

We summarize the steps for the proof of this theorem and its use
in effective computations.
The key tool is the ring
$\NKd(M)$ of the
original Hamiltonian $T$-space $M$.
As a vector space $\NKd(M) = \bigoplus_{t\in T} K_T(M^t)$, where $M^t$
consists of fixed points of the $t$ action on $M$, $t\in T$.
For each $t$, $K_T(M^t)$ may be computed  using well-known methods in
equivariant topology (see e.g. \cite{GKM, HHH05}).
We may also apply
the ordinary $K$-theoretic Kirwan surjectivity theorem,
which states that the map $\kappa_t$ induced by inclusion from
$K^*_T(M^t)$ to $K^*_T(\Phi^{-1}(0)
\cap M^t)$ is a surjection \cite{HarLan07}.
However, the ring structure
on $\NKd(M)$ is not the obvious one on $\bigoplus_{t\in T}K_T(M^t)$.
Thus the main technical
challenge,
as was the case in \cite{GHK05}, is to prove that the 
Kirwan map $\kappa$~\eqref{eq:kappa-in-intro} is indeed a
ring homomorphism.

An additional benefit of our point of view is that $\NKd(M)$ surjects
onto the full orbifold $K$-theory of any of its orbifold symplectic
quotients (at any regular value $\alpha$). This is because the isotropy information for every orbifold symplectic quotient $[M/\!/_{\alpha} T]$ is contained in
the ring $\NKd(M)$ when $M$ is a Hamiltonian $T$-space. 
As an illustration of our surjectivity theorem, in Section~\ref{sec:examples} we
calculate the orbifold cohomology of 
those
weighted projective spaces 
that occur as a symplectic quotient of $\C^n$ by a linear $S^1$ action.
We will discuss symplectic toric orbifolds in greater detail in a subsequent paper.

\medskip
\noindent{\bf Acknowledgements.} The authors thank Dan Dugger and Gregory D. Landweber for
many helpful conversations. The first author was supported in part by an NSF Grant
DMS-0606869, the Ruth Michler Award of the AWM, and a George Mason University Provost's Seed Grant.
The second author was supported in part by an NSERC Discovery Grant, an NSERC University Faculty
Award, and an Ontario Early Researcher Award. The third author was supported by NSF Grant
DMS-0835507, and NSF Advance funds from Cornell University.

\section{The inertial $K$-theory of a stably 
complex $T$-space}\label{sec:definition}

Throughout, $T$ will denote a compact connected abelian Lie group (i.e.\ a compact torus).
In this section, we define a new ring, the {\bf inertial $K$-theory}, associated to a stably almost
complex $T$-manifold $M$. The definition is similar in spirit to that of the inertial cohomology
associated to $M$ as in \cite{GHK05}. When $M$ is a Hamiltonian $T$-space, this inertial $K$-theory
ring gives rise to a surjective ring homomorphism onto the full orbifold $K$-theory defined by
\cite{JarKauKim07} of the symplectic quotient $M \mod_{\! \alpha} T$.

\subsection{The definition of inertial $K$-theory}
We begin by defining the inertial K-theory additively as a $K_T$-module. Suppose $M$ is a stably
complex $T$-space in the sense of \cite{GHK05}. Since $T$ is abelian, each  $M^t$ is a
$T$-space for $t\in T$.
\begin{definition}\label{def:inertial-Kthy}
The \textbf{inertial K-theory $NK^{\diamond}_T(M)$ of a stably 
complex $T$-space $M$} is
defined, as a $K_T(\pt)$-module, as
\begin{equation}\label{eq:def-inertial}
\NKd(M) := \bigoplus_{t \in T} K_T(M^t).
\end{equation}
\end{definition}
The grading $\diamond$ is with respect to group elements of $T$; as we will show in
Definition~\ref{def:productNK_T}, the product of a homogeneous $t_1$-class and a homogeneous
$t_2$-class is a homogeneous $(t_1\cdot t_2)$-class.  Here $K_T(X)$ denotes the integral
$T$-equivariant $K$-theory of Atiyah and Segal \cite{Seg68}.
In the case when $T$ acts with finite stabilizers, $\NKd(M)$ coincides
with \cite[Definition 2.1]{BecUri07} and with the full orbifold
$K$-theory of \cite{JarKauKim07} by \cite[Section 2]{BecUri07}.

We now proceed with the definition of the product on $\NKd(M)$ which
follows that of \cite{BecUri07} and \cite{GHK05}.
We begin with some observations about the normal bundle to the fixed point set $M^H$ for a fixed
subgroup $H\subset T$; these lead to a simple description (using {\bf logweights}) of the
{\bf obstruction bundle} with respect to which we twist the product. Let $V$ be a connected component of $M^H$ and let \(p \in V\). The linear
action by $H$ on each fiber $T_pM \oplus \tau$ (where $\tau$ is the trivial stabilization in the stably complex structure) is the identity on precisely $T_pV \oplus \tau$,
so there is an induced linear action of $H$ on the normal bundle \(\nu(V, M) := TM/TV \cong
\left(TM \oplus \tau\right)/\left(TV \oplus \tau\right)\) over $V\subset M^H$. Since $H$ preserves
the almost complex structure on $TM \oplus \tau$, it follows that $\nu(V, M)$ is a complex vector
bundle. The (complex) rank of $\nu(V,M)$ may vary as $V$ varies over the components of $M^H$.

We may write
\begin{equation}\label{eq:normal-bundle-t-decomp}
\nu(V, M) \cong \bigoplus_{\alpha\in\widehat{H}} \nu(V, M)_{\alpha},
\end{equation}
where $\hat{H}$ denotes the character group of $H$ and $\nu(V,
M)_{\alpha}$ is the subbundle of $\nu(V,M)$ on which $H$ acts by
weight $\alpha$. Each element $t\in H$ acts on a fiber
$\nu_p(V,M)_\alpha$ with eigenvalue $exp(2\pi i a_\alpha(t))$,
for a constant
$a_\alpha(t)\in [0,1)$ called the {\bf logweight} of $t$ on
$\nu(V,M)_\alpha$.

Now let
\[
\Mt := \bigsqcup_{(t_1,t_2)\in T \times T} M^{ t_1,t_2 }
\]
where $M^{ t_1,t_2 }$ is the submanifold of $M$ consisting of points fixed
simultaneously by $t_1$ and $t_2$.
For each pair $(t_1,t_2)\in T\times T$, let
$\mathcal{O}_{t_1,t_2}$ be the union over all connected components $V$ of
$M^{t_1, t_2}$ of $\mathcal{O}_{t_1,t_2}|_V$, where
\begin{equation}\label{eq:def-obstruction}
\mathcal{O}_{t_1,t_2}\vert_V := \bigoplus_{a_\alpha(t_1)+a_\alpha(t_2)+a_\alpha((t_1t_2)^{-1}) = 2}
\nu(V, M)_{\alpha}
\end{equation}
is the complex subbundle of $\nu(V,M)$ given by the component on which
the sum of the logweights of $t_1, t_2, (t_1t_2)^{-1}$ is 2.

\begin{definition}\label{def:obstruction}
The \textbf{obstruction bundle} $\mathcal{O}\to \Mt$ is the (disjoint) union of the bundles
$\mathcal{O}_{t_1, t_2}$ in~\eqref{eq:def-obstruction} for all pairs $t_1, t_2 \in T$, i.e.
$$
\mathcal{O} := \bigsqcup_{\stackrel{(t_1,t_2)\in T\times T}{V\ c.c.\ M^{t_1,t_2}}} \mathcal{O}_{t_1,t_2}\vert_V = \bigsqcup_{(t_1,t_2)\in T\times T} \mathcal{O}_{t_1,t_2},
$$
where ``c.c.'' denotes ``connected component of''.

Let $\epsilon(\mathcal{O}_{t_1,t_2}|_V):= \lambda_{-1}(\mathcal{O}_{t_1,t_2}|_V^*)$ denote the
$T$-equivariant $K$-theoretic Euler class of this bundle $\mathcal{O}_{t_1,t_2}|_V\to V$. We define
the \textbf{virtual fundamental class of $\Mt$} to be the sum
$$
\epsilon(\mathcal{O}) := \sum_{\stackrel{(t_1,t_2)\in T\times T}{V\ c.c.\ M^{t_1,t_2}}} \epsilon(\mathcal{O}_{t_1,t_2}|_V).
$$

\end{definition}

We may now define the product on $\NKd(M)$. By extending linearly, it clearly suffices to define the
product $b_1 \kt b_2$ of two { homogeneous} classes \(b_1 \in NK^{t_1}_T(M) = K_T(M^{t_1}), b_2 \in
NK^{t_2}_T(M) = K_T(M^{t_2}).\) Let $e_j: M^{t_1, t_2} \into M^{t_j}$ denote
the canonical $T$-equivariant inclusion map for any $t_1,t_2 \in T$ and \(j = 1,2,\) and let
$\ec_3: M^{ t_1,t_2}\into M^{t_1t_2}$ denote the inclusion map into points fixed by the product
$t_1t_2$.

\begin{definition}\label{def:productNK_T}
Let $M$ be a stably 
complex $T$-manifold and $\NKd(M)$ its inertial K-theory. Let \(t_1, t_2
\in T.\) The $\kt$ \textbf{product on the inertial K-theory $\NKd(M)$} is defined, for $b_1 \in
NK_T^{t_1}(M)$ and $b_2 \in NK_T^{t_2}(M)$, by
\begin{equation}\label{eq:ktproduct}
b_1\kt b_2 = \ec_{3, !}( e_1^* b_1\cdot e_2^* b_2\cdot \epsilon(\mathcal{O}))
\end{equation}
where $\cdot$ denotes the usual product in the $T$-equivariant $K$-theory $K_T(M^{t_1, t_2})$. By
extending linearly, $\kt$ is defined on all of $\NKd(M)$.
\end{definition}
Note that for $b_1\in NK_T^{t_1}(M)$ and $b_2\in NK_T^{t_2}(M)$, the product $b_1\kt b_2\in
NK_T^{t_1t_2}(M)$ is by definition a homogeneous class in the $t_1 t_2$-summand of $\NKd(M)$.

It is straightforward to show that $\NKd(M)$ is a $K_T(\pt)$-algebra.

\begin{proposition}\label{prop:NKring}
Let $M$ be a stably 
complex $T$-manifold. Then $\NKd(M)$ is a
commutative, associative, unital algebra over the ground ring
$K_T(\pt)$ with the multiplication $\kt$ of Definition~\ref{def:productNK_T}.
\end{proposition}

\begin{remark}\label{remark:integral-lift} 
As observed in \cite[Section 2]{BecUri07}, if $T$ acts on
$M$ locally freely
then $\NKd(M)\otimes\Q $ is isomorphic as an algebra to the full
orbifold K-theory
\cite{JarKauKim07} of $M$, $\okf([M/T])$. More specifically, the natural map
$\NKd(M) \to \NKd(M) \otimes \Q$ is an integral lift of
$\okf([M/T])$ as rings.
\end{remark}

In view of the proposition and remark above, we will henceforth occasionally
abuse language and refer to $\NKd(M)$ as being isomorphic to
$\okf([M/T])$. The precise statement is that when $T$ acts locally freely on $M$, $\NKd(M)$ is an
integral lift of $\okf([M/T])$.

\begin{proof}
The facts that $\NKd(M)$ is commutative and that $1 \in K_T(M^{id})$
acts as the unit element is immediate from the
definition~\eqref{eq:ktproduct}. The $K_T(\pt)$-algebra structure is
also immediate from the corresponding structure on each summand. It
remains to show associativity. Although we are not assuming that the
$T$-action on $M$ is locally free, the proof of associativity in our
$T$-equivariant situation nevertheless follows precisely that of
\cite{BecUri07} (and \cite{JarKauKim07}), so we do not reproduce it here.
\end{proof}

\subsection{The product on the fixed point set}\label{subsec:fixed-points}

Localization is a standard technique in equivariant topology. Given $M$ a stably 
complex
$T$-manifold and a $T$-equivariant algebraic/topological invariant, it is natural to ask whether
the invariant is encoded in terms of the $T$-fixed point set $M^T$ and local $T$-isotropy data
near the fixed points. The purpose of this section is to develop some
inertial $K$-theoretic analogues of standard localization theory in equivariant topology. The
presence of the quantum correction complicates matters, so we begin by defining a new ring
structure, denoted by $\star$, on $K_T(M^T)\otimes \Z[T]$,  where $\Z[T]$ is the group ring on $T$.
This is a K-theoretic version of the $\star$ product on $H_T^*(M^T)\otimes \Z[T]$ introduced in
\cite{GHK05}. When $M$ is a Hamiltonian $T$-space, we show in Section~\ref{subsec:Hamiltonian} that
the inertial $K$-theory injects into $K_T(M^T)\otimes \Z[T]$ as a ring, much as ordinary
equivariant $K$-theory $K_T(M)$ injects into $K_T(M^T)$ in such a case. This is the main motivation
for the product $\star$: the new product provides a different means of computing the product given
in \eqref{eq:ktproduct}.

For simplicity, we assume throughout that $M^T$ has finitely many connected components. In this
case
$$
K_T(M^T)\otimes \Z[T] = \bigoplus_{W\ c.c.\ M^T} (K_T(W)\otimes \Z[T]),
$$
where the direct sum is taken over connected components of $M^T$. When we refer to the restriction
of a class in $K_T(M^T)\otimes \Z[T]$ to a connected component $W$, we mean the summand
corresponding to $W$. As in the $\kt$ case, it suffices to define the $\star$ product of two
homogeneous classes \(\sigma_1 \otimes t_1\) and \(\sigma_2 \otimes t_2\) in \(K_T(M^T) \otimes
\Z[T],\) where \(t_1, t_2 \in T\) and \(\sigma_1, \sigma_2 \in K_T(M^T).\) Moreover, it also
suffices to specify the value of the product restricted to each connected component $W$ of $M^T$.

\begin{definition}\label{def:starproduct}
 Let
$\sigma_j\otimes t_j\in K_T(M^T)\otimes\Z[T]$ for $j=1,2$, where $t_j\in T$, be two homogeneous
classes. The {\bf $\star$ product on $K_{T}(M^{T}) \otimes \Z[T]$} is defined by
 \begin{equation}\label{eq:starproduct}
(\sigma_1\otimes t_1)\star (\sigma_2\otimes t_2)|_W :=
\left[ \sigma_1|_W\cdot \sigma_2|_W\cdot
\prod_{I_\mu\subset \nu(W,M)} \epsilon(I_\mu)^{a_\mu(t_1)+a_\mu(t_2)-a_\mu(t_1t_2)} \right] \otimes t_1 \cdot t_2,
\end{equation}
for each connected component $W$ of $M^{T}$. Here $\epsilon(I_\mu)\in K_T(W)$
denotes the $T$-equivariant K-theoretic Euler class of $I_\mu$, and $\nu(W,M)$ denotes the
normal bundle of $W$ in $M$. By extending linearly over the group algebra $\Z[T]$, the $\star$ product is defined on all of $K_{T}(M^{T}) \otimes \Z[T]$.
\end{definition}

The proof that the $\star$ product is associative is straightforward and identical to that of
\cite[Theorem 2.3]{GHK05}, so we do not repeat the argument, but record the result here.

\begin{theorem}
Let $M$ be a stably 
complex $T$-manifold.
 The multiplication $\star$ of Definition~\ref{def:starproduct} makes $K_T(M^T)\otimes \Z[T]$ into a
commutative, associative, unital algebra over the ground ring $K_T(\pt)$.
\end{theorem}

As stated above, the motivation for introducing $K_T(M^T) \otimes \Z[T]$ with the $\star$ product
is to allow for a kind of ``localization'' theorem in inertial $K$-theory. Consider the inclusion
$i_t:M^T\hookrightarrow M^t$ for each $t\in T$. This induces a map $i_t^*:K_T(M^t)\rightarrow
K_T(M^T)\otimes t\subset K_T(M^T)\otimes \Z[T]$. Combining the $i_t^*$ for each \(t \in T,\) we
obtain a map of $K_T$-modules
\begin{equation}\label{eq:KTmodule}
i_{NK}^*: NK_T(M)\longrightarrow K_T(M^T)\otimes \Z[T],
\end{equation}
which we refer to as a ``restriction map,'' since it is induced by the geometric inclusion \(M^T
\into M^t\) for each \(t \in T.\) This morphism $i_{NK}^*$ is in fact a ring homomorphism with
respect to the $\kt$ and the $\star$ products.

\begin{theorem}\label{th:homomorphism}
Let $M$ be a stably 
complex $T$-manifold. Let
$K_T(M^T)\otimes\Z[T]$ be endowed with the product $\star$ of
Definition~\ref{def:starproduct}. Then the restriction map
$\iNK^*:\ (\NKd(M), \kt) \to (K_T(M^T)\otimes\Z[T], \star)$ is a $K_{T}(\pt)$-algebra homomorphism.
\end{theorem}

\begin{proof}

This proof follows that of Theorem 3.6 of \cite{GHK05},
though the explanation below is slightly different from that in
\cite{GHK05} and is self-contained.

We begin by noting that for any \(t_1, t_2 \in T,\) the exponent
$a_\mu(t_1)+a_\mu(t_2)-a_\mu(t_1t_2)$ appearing in the definition of the $\star$ product is always
either 0 or 1. Using the defining properties of Euler classes, we may deduce
$$
\prod_{I_\mu\subset \nu(W, M)} \epsilon(I_\mu)^{a_\mu(t_1)+a_\mu(t_2)-a_\mu(t_1t_2)} = \epsilon\left(\bigoplus_{\stackrel{I_\mu\subset \nu(W, M)}{a_\mu(t_1)+a_\mu(t_2)-a_\mu(t_1t_2)=1}} \!\!\!\!\!\!\!\!\!\!\!\! I_\mu\ \ \ \ \right),
$$
so the expression appearing in the definition of the $\star$ product
is the Euler class of a certain sub-bundle of $\nu(W, M)$. We will now
show that this sub-bundle has a different description.
Given $\mu$ such that
$a_\mu(t_1)+a_\mu(t_2)-a_\mu(t_1t_2)=1$, then either $a_\mu(t_1t_2)\neq 0$, in which case
\begin{align*}
a_\mu(t_1)+a_\mu(t_2)-a_\mu(t_1t_2) &= a_\mu(t_1)+a_\mu(t_2)-(1-a_\mu((t_1t_2)^{-1})=1\\
& \implies  a_\mu(t_1)+a_\mu(t_2)+a_\mu((t_1t_2)^{-1})=2\\
&\implies I_\mu\subset \mathcal{O}_{t_1,t_2}|_W,
\end{align*}
or else $a_\mu(t_1t_2)=0$, in which case $I_\mu\subset
\nu(M^{t_1,t_2}, M^{t_1t_2})|_W$. Conversely, if $I_\mu$ is an
isotypic component of $\mathcal{O}_{t_1,t_2}|_W\oplus \nu(M^{t_1,t_2},
M^{t_1t_2})|_W$, then a similar simple argument shows that
$a_\mu(t_1)+a_\mu(t_2)-a_\mu(t_1t_2)=1$. Thus we have shown that
\begin{equation}\label{eq:obstructionequivalence}
\bigoplus_{\stackrel{I_\mu\subset \nu(W, M)}{a_\mu(t_1)+a_\mu(t_2)-a_\mu(t_1t_2)=1}}\!\!\!\!\!\!\!\!\!\!\!\! I_\mu\ \ =\  \ \mathcal{O}_{t_1,t_2}|_W \oplus \nu(M^{t_1,t_2}, M^{t_1t_2})|_W.
\end{equation}
This makes it clear that the obstruction bundle over$ M^{t_1,t_2}$
enters into the $\star$ product given by \eqref{eq:def-obstruction} as
well.

To prove that $i^*_{NK}$ is a ring homomorphism, it suffices to check the statement for homogeneous
elements. Let $b_1\in K_T(M^{t_1})$ and $b_2\in K_T(M^{t_2})$. Then for each fixed point component
$W$ of $M^T$, we have
\begin{align*}
i^*_{NK}(b_1\kt b_2)|_W &= i^*_{NK}[(\overline{e}_3)_*(e_1^*b_1\cdot e_2^*b_2\cdot \epsilon(\mathcal{O}_{t_1,t_2}))]|_W\\
& = b_1|_W \cdot b_2|_W\cdot  i^*_{NK}[(\overline{e}_3)_*(\epsilon(\mathcal{O}_{t_1,t_2}))]|_W\\
& =  b_1|_W \cdot b_2|_W \cdot \epsilon(\mathcal{O}_{t_1,t_2})|_W \epsilon(\nu(M^{t_1,t_2}, M^{t_1t_2}))|_W\\
& = b_1|_W \cdot b_2|_W\cdot \epsilon(\mathcal{O}_{t_1,t_2}|_W \oplus \nu(M^{t_1,t_2}, M^{t_1t_2})|_W).
\end{align*}
The equivalence~\eqref{eq:obstructionequivalence} of bundles allows us to conclude that this
product $b_1\kt b_2$ restricted to a component $W$ of $M^T$ agrees with the product of
$i_{NK}^*(b_1)|_W\star i_{NK}^*(b_2)|_W$ as in~\eqref{eq:starproduct}, as desired.
\end{proof}

\subsection{The case that $M$ is a Hamiltonian $T$-space}\label{subsec:Hamiltonian}

We now turn to the special case that $(M,\omega)$ is a Hamiltonian $T$-space.  The motivation for
the definition of the $\star$ product on $K_T(M^T) \otimes \Z[T]$ is to provide a target for a
localization theoreom. As in the case of ordinary equivariant (rational) cohomology and equivariant
$K$-theory, the fixed points $M^T$ of a Hamiltonian $T$-space play a special role in inertial
$K$-theory. In particular, we have the following theorem.

\begin{theorem}\label{th:injectivity}
Let $(M, \omega, \Phi)$ be a Hamiltonian $T$-space.
Suppose there exists a component of $\Phi$ which is proper and bounded
below, and further suppose that $M^T$ has only finitely many connected
components. The map of rings given by
$$
i_{NK}^*: (NK_T(M), \kt) \rightarrow (K_T(M^T)\otimes \Z[T], \star)
$$
is injective.
\end{theorem}

\begin{proof}
We have already shown in Theorem~\ref{th:homomorphism} that $i_{NK}^*$ is an algebra homomorphism.
Thus we need only show injectivity.  Since $T$ is abelian, standard symplectic techniques allow us
to conclude that $(M^t, \omega \vert_{M^t}, \Phi \vert_{M^t})$ is also a Hamiltonian $T$-space for
each $t\in M$. Furthermore, the second author and Landweber prove in \cite{HL-kernel} that, for any
Hamiltonian $T$-space $M$ satisfying the assumptions of our theorem, the inclusion
$M^T\hookrightarrow M^t$ induces an injection $K_T(M)\hookrightarrow K_T(M^T)$. Thus each map
$$
i_t^*: K_T(M^t)\longrightarrow K_T(M^T)\otimes t
$$
is an injection, which implies that the map
$i_{NK}^*$ defined by an $i_t^*$ on each component is also injective,
as desired.
\end{proof}

It is clear from the proof that the statement holds for any stably complex $T$-space $M$ with the property that there is an injection
$$
K_T(M^t)\longrightarrow K_T(M^T)
$$
for every $t\in T$. In the case of equivariant {cohomology}, these spaces are called {\bf robustly
equivariantly injective} in \cite{GHK05}. Thus, Theorem~\ref{th:injectivity} holds also in the case
that $M$ is robustly equivariantly injective in $K$-theory.

\section{Surjectivity from inertial $K$-theory to full orbifold $K$-theory}\label{sec:surjectivity}

The main motivation for our definition of the inertial $K$-theory of a $T$-space $M$ is that we can exploit
it to give explicit computations of the full orbifold $K$-theory of abelian symplectic quotients.
Our methods build on the Kirwan surjectivity techniques in ordinary (non-orbifold) $K$-theory as
developed in \cite{HarLan07}. In this section we prove a general surjectivity theorem onto the
\(\okf([M \mod_{\! \alpha} T])\) of abelian symplectic quotients and discuss in some detail the computation of
the kernel of the surjection. Indeed, for a wide class of examples, this method yields an explicit
description via generators and relations of $\okf([M \mod_{\! \alpha} T])$. In Section~\ref{sec:examples} we
will give an explicit illustration of the use of our techniques in the case of weighted projective
spaces 
occurring as symplectic quotients. 

We take a moment here to discuss the technical hypotheses on the
moment map $\Phi$ to be used in this section.  A more
detailed discussion of these hypotheses may be found in
\cite[beginning of Section 3]{HL-kernel}. For
the surjectivity theorems (Theorems~\ref{theorem:orbi-Kirwan-surj}
and~\ref{theorem:orbi-Kirwan-surj-Gamma}) it is only necessary to
assume that $\Phi$ is proper; this ensures that the tools of
equivariant Morse theory may be applied to $\|\Phi\|^2$.
However, for the computation of the kernel of the surjective
Kirwan map as given in Theorem~\ref{theorem:kernel}, we need
the additional technical assumption that there is a component of
$\Phi$ that is proper and bounded below, and that $M^T$ the fixed
point set has only finitely many connected components. In practice
this is not a very restrictive condition (see \cite[Section
3]{HL-kernel} for further discussion).

\subsection{Surjectivity to $\okf([\Phi^{-1}(\alpha)/T])$}

Here we show that under some mild technical assumptions,
the inertial $K$-theory associated to a
Hamiltonian $T$-space $(M,\omega,\Phi)$ defined in the previous
section surjects onto the full orbifold $K$-theory of the symplectic
quotient \([M \mod_{\! \alpha} T] := [\Phi^{-1}(\alpha)/T],\) where $\alpha$
is a regular value of $\Phi$.

We begin by recalling the setup. As above, $T$ denotes a compact connected abelian Lie group (i.e. a connected compact torus). Let $(M,\omega,\Phi)$ be a finite-dimensional Hamiltonian
$T$-space,  and assume that the moment map $\Phi: M \to \t^*$ is proper. Suppose that \(\alpha \in
\t^*\) is a regular value of $\Phi$. Then the level set \(\Phi^{-1}(\alpha)\) is a submanifold of
$M$, and $T$ acts locally freely on $\Phi^{-1}(\alpha)$. Hence the quotient
\([\Phi^{-1}(\alpha)/T]\) is an orbifold. By results of Becerra and
Uribe,
\begin{equation}\label{eq:def-orb-Ktheory}
\okf([\Phi^{-1}(\alpha)/T]) \cong NK^{\diamond}_{T}(\Phi^{-1}(\alpha)),
\end{equation}
where by the left hand side we mean (by slight abuse of notation) an integral lift of the orbifold $K$-theory, as in Remark~\ref{remark:integral-lift}.

We now show that the right-hand side of \eqref{eq:def-orb-Ktheory} is computable, using techniques
from equivariant symplectic geometry. The inclusion of the level set
\[
\iota: \Phi^{-1}(\alpha) \into M
\]
is $T$-equivariant and induces a map in equivariant $K$-theory,
\[
\iota^*: K_T(M) \to K_T(\Phi^{-1}(\alpha)).
\]
The fixed points sets $M^t$ are also Hamiltonian $T$-spaces, so we also have
\[
\iota^*: K_T(M^t) \to K_T(\Phi^{-1}(\alpha)^t).
\]
Here by abuse of notation we denote also by $\iota$ the inclusions \(\Phi^{-1}(\alpha)^t \into
M^t\), for all \(t \in T\). Hence there exists a map (also denoted $\iota^*$)
\begin{equation}\label{eq:iota-on-NK}
\iota^*: NK_T(M) = \bigoplus_{t \in T} K_T(M^t) \to NK_T(\Phi^{-1}(\alpha)) = \bigoplus_{t \in T} K_T(\Phi^{-1}(\alpha)^t)
\end{equation}
defined in the obvious way on corresponding summands. It is immediate from the definition that the map is an additive homomorphism.

What is not at all obvious is that the map $\iota^*$ of~\eqref{eq:iota-on-NK} is also a {ring}
homomorphism with respect to the multiplicative structure $\kt$ on both sides, and furthermore that
$\iota^*$ provides an effective way of computing $NK_T(\Phi^{-1}(\alpha))$, and hence
$\okf([\Phi^{-1}(\alpha)/T])$. We now address each of these issues in turn.

As observed in \cite{GolHar06} and \cite{GHK05}, it is {not} necessarily true that a general
$T$-equivariant map of $T$-spaces \(f: X \to Y\) induces a ring map $f^{*}: NH^{*}_{T}(Y) \to
NH^{*}_{T}(X)$ on inertial cohomology, with respect to the $\smile$ product in inertial cohomology.
Nevertheless, if a $T$-equivariant inclusion \(\iota: X \into Y\) behaves well with respect to the
fixed point sets, \cite[Proposition 5.1]{GHK05} states that $\iota^*$ { is} a ring homomorphism
with respect to the new product in inertial cohomology. We now prove a $K$-theoretic analogue of
this fact.

\begin{proposition}\label{prop:ring-map}
Let $Y$ be a stably complex $T$-space. Let \(\iota: X \into Y\) be a $T$-equivariant inclusion, and suppose also that $X$ is tranvserse to all \(Y^{t}, t \in T.\) Then the map
\(\iota^{*}: NK^{\diamond}_{T}(Y) \to NK^{\diamond}_{T}(X)\)  induced by inclusion is a ring
homomorphism with respect to $\kt$.
\end{proposition}

\begin{proof}
The argument is nearly the same as that given for equivariant
cohomology \cite[Proposition 5.1]{GHK05}, so we do not fully reproduce it
here. The only additional item to check in our $K$-theoretic setting is that
the following diagram in $T$-equivariant $K$-theory
\begin{equation}\label{eq:pushpull-commute}
\xymatrix{
K_T(Y^{t_1, t_2})  \ar[r]^{(\overline{e}_3^{1,2})_{!}} \ar[d]^{\iota^*} & K_T(Y^{t_1 t_2}) \ar[d]^{\iota^*}  \\
K_T(X^{t_1, t_2}) \ar[r]^{(\overline{e}_3^{1,2})_{!}}  & K_T(X^{t_1 t_2})
}
\end{equation}
commutes, i.e. \(\iota^* (\overline{e}_{3}^{1,2})_{!} = (\overline{e}_{3}^{1,2})_{!} \iota^*,\)
where by abuse of notation we denote by $\iota^*$ both of the inclusions \(X^{t_1 t_2} \into Y^{t_1
t_2}\) and \(X^{t_1, t_2} \into Y^{t_1, t_2},\) and also by $\overline{e}^{1,2}_{3}$ both of the
inclusions \(X^{t_1, t_2} \into X^{t_1 t_2}\) and \(Y^{t_1, t_2} \into Y^{t_1 t_2}.\) To see this,
note that both normal bundles in question have natural $T$-equivariant complex structures, so they
have canonical spin-c structures, as remarked in Section~\ref{sec:definition}. The definition of
the pushforward in equivariant $K$-theory uses the Thom isomorphism with respect to a spin-c
structure on the relevant normal bundles to an inclusion.  Hence, in order to check
that~\eqref{eq:pushpull-commute} commutes, it suffices to check that the normal bundle
\(\nu(Y^{t_{1}, t_{2}}, Y^{t_{1}t_{2}})\) restricts precisely to the normal bundle \(\nu(X^{t_{1},
t_{2}}, X^{t_{1}, t_{2}})\) via \(\iota^{*}.\) This follows from the transversality of $X$ to all
fixed points $Y^{t}$, for any $t \in T$.
\end{proof}

Although the transversality hypothesis of Proposition~\ref{prop:ring-map} is rather
restrictive, there is a natural class of examples in which the
conditions are satisfied.

\begin{lemma}\label{lemma:level-transverse}
Let $(M, \omega, \Phi)$ be a Hamiltonian $T$-space with proper moment map
\(\Phi: M \to \t^{*}.\) Assume that $\alpha \in \t^{*}$ is a regular
value of $\Phi$. Then the inclusion of the level set \(\iota:
\Phi^{-1}(\alpha) \into M\) satisfies the hypotheses of
Proposition~\ref{prop:ring-map}.
\end{lemma}

This is a purely topological statement and its proof can be found in \cite[Theorem 6.4]{GHK05}.
We may conclude that we have a ring homomorphism onto the integral
full orbifold $K$-theory
\begin{equation}\label{eq:orbi-Kirwan-map}
\xymatrix{
\kappa: NK^{\diamond}_T(M) \ar[r]^{\iota^*} & NK^{\diamond}_T(\Phi^{-1}(\alpha))
\ar[r]^<<<<<{\cong} & \okf([\Phi^{-1}(\alpha)/T]).
}
\end{equation}
We refer to this composition $\kappa$ as the {\bf
orbifold Kirwan map}. We now recall the following, which is the
analogue of the Kirwan surjectivity theorem (originally proved for
rational Borel-equivariant cohomology) in integral $K$-theory
\cite[Theorem 3.1]{HarLan07}. We state the theorem only in the special case
needed here.

\begin{theorem}\label{theorem:ordinary-Kirwan} (\cite{HarLan07})
Let $(M, \omega)$ be a Hamiltonian $T$-space with proper moment map \(\Phi: M \to \t^{*}.\) Assume that $\alpha \in \t^{*}$ is a regular value of $\Phi$. Then the map $\iota^{*}$ induced by the inclusion \(\iota: \Phi^{-1}(\alpha) \into M,\)
\[
\iota^{*}: K_{T}(M) \to K_{T}(\Phi^{-1}(\alpha)),
\]
is a surjection.
\end{theorem}

The following is now straightforward.

\begin{theorem}\label{theorem:orbi-Kirwan-surj}
Let $(M, \omega, \Phi)$ be a Hamiltonian $T$-space with proper moment
map \(\Phi: M \to \t^{*}.\) Assume that $\alpha \in \t^{*}$ is a
regular value of $\Phi$, so that $T$ acts locally freely on
$\Phi^{-1}(\alpha)$. Then the orbifold Kirwan map to the
orbifold $K$-theory
\[
\kappa: NK^{\diamond}_{T}(M) \to
\okf([\Phi^{-1}(\alpha)/T])
\]
defined in~\eqref{eq:orbi-Kirwan-map}
is a surjective ring homomorphism.
\end{theorem}

\begin{proof}
Proposition~\ref{prop:ring-map} and Lemma~\ref{lemma:level-transverse}
guarantee that $\kappa$ is a ring homomorphism,
so it is enough to check that $\iota^*$ is surjective.
Since $\iota^*$ in~\eqref{eq:iota-on-NK} is defined
separately on each summand, the surjectivity follows from surjectivity
on each summand. We observe that for a Hamiltonian
$T$-space $(M, \omega, \Phi)$ as given, for any $t\in T$, the fixed set
\(M^t\)  is itself a Hamiltonian $T$-space with moment map the restriction $\Phi
\vert_{M^t}$. In particular, since \(\Phi^{-1}(\alpha)^t = (\Phi
\vert_{M^t})^{-1}(\alpha),\) Theorem~\ref{theorem:ordinary-Kirwan}
implies that each
\[
\iota^*: K_T(M^t) \to K_T(\Phi^{-1}(\alpha)^t)
\]
is surjective, completing the proof.
\end{proof}

Thus, in order to compute $\okf([\Phi^{-1}(\alpha)/T])$, we must explicitly compute
$NK^{\diamond}_{T}(M)$ and  identify the kernel of $\kappa$. We discuss the kernel in
Section~\ref{subsec:kernel}. As for the domain $NK^{\diamond}_{T}(M)$, we observe that in fact a
different, smaller, ring already surjects onto $NK^{\diamond}_{T}(\Phi^{-1}(\alpha))$. This is
highly relevant for computations, since the domain of the orbifold Kirwan
map~\eqref{eq:orbi-Kirwan-map} is an { infinite} direct sum, while the smaller subring is a {
finite} direct sum (and thus more manageable for explicit computations).

The essential idea is similar to those already developed in \cite{GHK05} and subsequently in
\cite{GolHar06}. Since $T$ acts locally freely on $\Phi^{-1}(\alpha)$ and because
$\Phi^{-1}(\alpha)$ is compact (since $\Phi$ is proper), there are only finitely many orbit types,
and hence only finitely many elements \(t \in T\) occur in the stabilizer of a point in
$\Phi^{-1}(\alpha)$. In particular, $\Phi^{-1}(\alpha)^{t}$ is non-empty for only finitely many
$t$. Thus in the codomain of the map~\eqref{eq:iota-on-NK}, only finitely many summands are
non-zero. It is straightforward to see that the restriction of $\kappa$ to the direct sum
\begin{equation}\label{eq:smaller-codomain}
\bigoplus_{t:\ \Phi^{-1}(\alpha)^{t} \neq \emptyset} K_{T}(M^{t})
\end{equation}
itself surjects onto $NK^{\diamond}_{T}(\Phi^{-1}(\alpha))$.
Unfortunately~\eqref{eq:smaller-codomain} is not closed under the $\kt$ multiplication on
$NK^{\diamond}_{T}(M)$. Hence we introduce the {\bf $\Gamma$-subring of $NK_{T}^{\diamond}(M)$},
which is the smallest subring containing~\eqref{eq:smaller-codomain}; this will also surject onto
$NK^{\diamond}_{T}(\Phi^{-1}(\alpha))$.

Recall that if a torus $T$ acts locally freely on a space $Y$, then by
definition the
stabilizer group $\Stab(y)$ of any point \(y \in Y\) is finite; we call
$\Stab(y)$ a {\bf finite stabilizer group}. Similarly, given a finite
stabilizer group $\Stab(y)$, we call an element \(t \in \Stab(y)\) a
{\bf finite stabilizer element}. Let $\Gamma$ denote the subgroup of
$T$ generated by all finite stabilizer elements, called the {\bf
finite stabilizer subgroup} of $T$ associated to $Y$.  For any subgroup
$\Gamma$ of $T$ we may define the following.

\begin{definition}
Let $Y$ be a stably 
complex $T$-space and let $\Gamma$ be a subgroup of $T$. Then
\begin{equation}\label{eq:def-Gammasubring}
NK^{ \Gamma}_{T}(Y) := \bigoplus_{t \in \Gamma} K_{T}(Y^{t})
\end{equation}
is a subring of $NK^{\diamond}_{T}(Y)$, called the {\bf $\Gamma$-subring} of $NK^{\diamond}_{T}(Y)$.
\end{definition}

\begin{remark}
The $\Gamma$-subring is closed under the $\kt$ multiplication.  This
follows immediately from Definition~\ref{def:productNK_T} and the
comment after it, together with the fact that $\Gamma$ is a subgroup
and thus closed under multiplication.
\end{remark}

In the case of the level set $\Phi^{-1}(\alpha)$ of a Hamiltonian $T$-action $(M, \omega, \Phi)$
with a proper moment map $\Phi$, this associated subgroup $\Gamma$ will be a finite subgroup of
$T$.\footnote{Although we do not make it explicit in the notation, this subgroup $\Gamma$ depends on the choice of level set $\Phi^{-1}(\alpha)$. In
\cite{GHK05} the subgroup $\Gamma$ is chosen such that $\kappa^{\Gamma}$ is surjective for any
choice of level set, so our choice differs slightly from that of Goldin, Holm, and Knutson.}  Hence
the direct sum in~\eqref{eq:def-Gammasubring} is also finite. The preceding discussion establishes
the following.

\begin{theorem}\label{theorem:orbi-Kirwan-surj-Gamma}
Let $(M, \omega, \Phi)$ be a Hamiltonian $T$-space with proper moment
map \(\Phi: M \to \t^{*}.\) Assume that $\alpha \in \t^{*}$ is a
regular value of $\Phi$, so that $T$ acts locally freely on
$\Phi^{-1}(\alpha)$. Let $\Gamma$ be the finite stabilizer subgroup of
$T$ associated to $\Phi^{-1}(\alpha)$. Then the
restriction $\kappa^\Gamma$ of the orbifold Kirwan
map~\eqref{eq:orbi-Kirwan-map} to the
$\Gamma$-subring~\eqref{eq:def-Gammasubring},
\[
\kappa^{\Gamma}: NK^{\Gamma}_{T}(M) \to NK^{\diamond}_{T}(\Phi^{-1}(\alpha)) \cong \okf([\Phi^{-1}(\alpha)/T]),
\]
is a surjective ring homomorphism.
\end{theorem}

\subsection{The kernel of the orbifold Kirwan map}\label{subsec:kernel}

Theorem~\ref{theorem:orbi-Kirwan-surj-Gamma} shows that the inclusion
of the level set \(\Phi^{-1}(\alpha) \into M\) induces a surjective
ring homomorphism \(\kappa^\Gamma\) from the $\Gamma$-subring
$NK_T^{\Gamma}(M)$ to $NK_T(\Phi^{-1}(\alpha)) \cong
\okf([\Phi^{-1}(\alpha)/T])$. So to compute explicitly the ring
$\okf([\Phi^{-1}(\alpha)/T])$, it then remains to compute the domain
$NK_T^{\Gamma}(M)$ and the kernel $\ker(\kappa^{\Gamma})$ of the
orbifold Kirwan map.

We begin with some comments on the computation of the domain
$NK^{\Gamma}_{T}(M)$. In a large class of examples,
the $T$-equivariant $K$-theory of the original Hamiltonian $T$-space
$M$ is well-known to be explicitly computable.
For example, following  the Delzant construction of symplectic toric orbifolds, the
original Hamiltonian space $M$ is just $\C^{N}$, an affine space
equipped with a linear $T$-action, so in particular is
$T$-equivariantly contractible. In this case, then, the
domain \(NK^{\Gamma}_{T}(M=\C^{N})\) is simply a direct sum of $|\Gamma|$ copies of $K_{T}(\pt) \cong
R(T)$, as an $R(T)$-module.

Another important class of examples are {\bf GKM spaces}.
Suppose that the original Hamiltonian $T$-space $(M, \omega, \Phi)$ is {\bf
GKM} in the sense of  \cite{HHH05} or \cite{HL-kernel}.
Each fixed point set $M^{t}$ is then also GKM.
A specific class of examples are the homogeneous spaces $G/T$
of compact connected Lie groups $G$ with maximal torus $T$, considered
as a Hamiltonian $T$-space with respect to the natural left action of
the maximal torus $T$ (see e.g. \cite[Lemma 8.2]{GHK05}).
In this situation, the results of Section~\ref{subsec:Hamiltonian} imply that the natural restriction
$$
i_{NK}^*: (NK^\diamond_T(M), \kt) \rightarrow (K_T(M^T)\otimes \Z[T], \star)
$$
is injective.
Furthermore, the image of
this injection can be explicitly and combinatorially described using
GKM (``Goresky-Kottwitz-MacPherson'') theory.  Indeed, GKM-type
techniques in equivariant cohomology were already used in
\cite[Section 8]{GHK05} in order to give explicit computations
associated to flag manifolds in inertial cohomology; the $K$-theoretic methods using
GKM theory in $K$-theory (as in \cite{HHH05, HL-kernel}, and
references therein) are analogous.

Now we turn to the computation of the kernel of the orbifold Kirwan map $\kappa$. First observe (as
in the proof of Theorem~\ref{theorem:orbi-Kirwan-surj}) that for each \(t \in \Gamma,\) $M^t$ is
itself a Hamiltonian $T$-space, with moment map the restriction of $\Phi$ to $M^t$. Thus, for each
\(t \in \Gamma,\) the map induced by inclusion
\[
\kappa_{t}:= \iota^*: K_{T}(M^{t}) \to K_{T}(\Phi^{-1}(\alpha)^{t})
\]
is precisely the ordinary (non-orbifold)
Kirwan map for  the
Hamiltonian $T$-space $(M^t, \omega \vert_{M^t}, \Phi
\vert_{M^t})$. The kernel of the ordinary Kirwan map for an abelian
symplectic quotient has been explicitly described in \cite{HL-kernel}
following previous work in cohomology of Tolman and Weitsman
\cite{TW03}.  More specifically, \cite[Theorem 3.1]{HL-kernel} gives a
list of ideals in $K_{T}(M^{t})$ which generates $\ker(\kappa_{t})$;
for reference, we include the statement below. We fix once and for all
a choice of inner product on $\t^*$, with respect to which we define
the norm-square $\|\Phi\|^2: M \to \R$ and identify \(\t \cong \t^*.\)

\begin{theorem}\label{theorem:kernel} (\cite[Theorem 3.1]{HL-kernel})
Let $(M, \omega, \Phi)$ be a Hamiltonian $T$-space. Suppose there
exists a component of $\Phi$ which is proper and bounded below, and
further suppose that $M^T$ has only finitely many connected
components. Let
\begin{equation}\label{eq:def-Z}
Z := \{ \Phi(C) \hspace{1mm} \vert \hspace{1mm} C \hspace{1mm} \mbox{a
connected component of} \hspace{1mm} \Crit(\|\Phi\|^2) \subseteq M \}
\subseteq \t^* \cong \t
\end{equation}
be the set of images under $\Phi$ of components of the critical set of
$\|\Phi\|^2$. For \(\xi \in \t,\) define
\begin{eqnarray*}
M_{\xi} &  := & \{ x \in M \, | \, \left<\mu(x), \xi\right> \leq 0 \}, \\
\Kernel_{\xi} & := & \{ \alpha \in K^*_T(M) \, | \,  \alpha |_{M_{\xi}} = 0 \}, \quad \mbox{and} \\
\Kernel & := & \sum_{\xi \in Z \subseteq \t} \Kernel_{\xi}.
\end{eqnarray*}
Then there is a short exact sequence
\[
\xymatrix{
0 \ar[r] & \Kernel \ar[r] & K_T^*(M) \ar[r]^-{\kappa} & K^*_T(\Phi^{-1}(0)) \ar[r] & 0,
}
\]
where \(\kappa: K^*_T(M) \to K^*_T(\Phi^{-1}(0))\) is the Kirwan map.
\end{theorem}

\begin{remark}
Although the statement of \cite[Theorem 3.1]{HL-kernel} explicitly
refers to the symplectic quotient $M \mod_{\! \alpha} T$, it is straightforward to
see that the theorem in fact holds at the level of the $T$-equivariant
$K$-theory of the level set $\Phi^{-1}(0)$. Moreover, although
\cite[Theorem 3.1]{HL-kernel} is (for convenience) stated only for the $0$-level set
$\Phi^{-1}(0)$, since $T$-moment maps
are determined only up to an additive constant, it is
straightforward to see that the analogous statement also holds for the case
of non-zero regular values \(\alpha \in \t^*\) and its corresponding
level set \(\Phi^{-1}(\alpha).\)
\end{remark}

Using Theorem~\ref{theorem:kernel} and localization
techniques as described  in \cite[Section
2]{HL-kernel}, an explicit list of generators for the kernel of
$\kappa_t$ may be constructed for each \(t \in \Gamma.\)
Since
\[
\ker(\kappa^\Gamma) = \bigoplus_{t \in \Gamma} \ker(\kappa_{t}),
\]
the kernel of the restricted orbifold Kirwan map $\kappa^\Gamma$ is therefore obtained by computing each $\ker(\kappa_t)$ separately. This completes the explicit description of $\okf([\Phi^{-1}(\alpha)/T])$ in terms of generators and relations for a wide class of abelian symplectic quotients.

\section{Example: the
full orbifold $K$-theory of weighted projective spaces}\label{sec:examples}

We now give a complete description of the ring structure of the full orbifold $K$-theory of weighted projective spaces obtained as symplectic quotients of $\C^n$ by $S^1$, closing with the worked example of $\P(1,2,4)$.  We do not require any mention of a stacky fan as in \cite{BCS05} or of labelled polytopes as in \cite{LT97}.  The technique here is different from that in \cite{BecUri07} because we use the fact that these spaces occur as symplectic quotients. We also avoid using the Chern character isomorphism, allowing us to obtain results over $\Z$.  In particular, we prove in Proposition~\ref{prop:torsionfree} that $\okf([M])$ has no additive torsion for those weighed projective spaces $[M]$ obtained as symplectic quotients by a (connected) circle.

Recall that such a weighted projective space is specified by an
integer vector in $\Z^{n+1}_{>0}$:
\[
b = (b_0, b_{1}, b_{2}, \cdots, b_{n}), \quad b_{k} \in \Z, b_{k} > 0.
\]
The vector $b$ determines an action of $S^1$ on $\C^{n+1}$, defined by
\[
t \cdot (z_0, z_1, \ldots, z_n) := (t^{b_0}z_0, t^{b_1}z_1, \ldots,
t^{b_n}z_n),
\]
for \(t \in S^1\) and \((z_0, z_1, \ldots, z_n) \in \C^{n+1}.\) An
$S^1$-moment map for this action is given by
\[
\Phi(z_0, z_1, \ldots, z_n) = - \frac{1}{2} \sum_{k=0}^n b_k
\|z_k\|^2.
\]
This is clearly proper and its negative bounded below. Moreover, since the $b_k$ are positive, the
only $S^1$-fixed point in $\C^{n+1}$ is the origin \(\{0\} \in \C^{n+1};\) in particular,
$(\C^{n+1})^{S^1}$ has only finitely many connected components. Any negative moment map value is
also regular, so we may define the {\bf weighted projective space} $\P^n_b$ (sometimes also denoted
$\P^n(b)$) as the orbifold arising as a symplectic quotient
\[
\P^n_b := \C^{n+1} \mod_{\! \alpha} S^1 = [\Phi^{-1}(\alpha)/S^1]
\]
for a regular (negative) value $\alpha$.  The differential structure does not change as $\alpha$ varies, though the symplectic volume does.  We note that $\P^n_b$ may be a { non-effective} orbifold: if the integers $b_i$ are not relatively prime, i.e. \(g = \gcd(b_0, b_1, \ldots, b_n) \neq 1,\) then there is a global stabilizer isomorphic to the cyclic group $\Z_g$. Effective or not, however, all the hypotheses of the theorems in Section~\ref{sec:surjectivity} are satisfied for these weighted projective spaces.

We begin by computing the finite stabilizer subgroup \(\Gamma \subseteq S^1\) for
$\Phi^{-1}(\alpha)$. Throughout, we will denote by $\Z_{\ell}$ the cyclic subgroup in $S^1$ given
by the $\ell$-th roots of unity \(\{e^{2\pi is/\ell}\ |\ s=0,\dots,\ell-1\}\) in $S^1$. Given a
non-zero vector \(z = (z_0, z_1, \ldots, z_n)\) in $\C^{n+1}$, the stabilizer subgroup of $z$ in
$S^1$ is precisely
\[
\Gamma_z := \bigcap_{z_k \neq 0} \Z_{b_k} \subseteq S^1.
\]
In particular, this implies that $\Gamma$ is generated by the
subgroups $\Z_{b_k}$ for each $k$, \(0 \leq k \leq n,\) so \(\Gamma =
\Z_{\ell} \subseteq S^1\) where
\[
\ell = \lcm(b_0, b_1, \ldots, b_n).
\]
Let \(\zeta_s := e^{2\pi i s/\ell} \in S^1.\) Then,
by definition, the $\Gamma$-subring of the inertial $K$-theory of the
$S^1$-space $\C^{n+1}$ is additively defined to be
\[
NK^{\Gamma}_{S^1}(\C^{n+1}) := \bigoplus_{s=0}^{\ell}
K_{S^1}((\C^{n+1})^{\zeta_s}) \cong \bigoplus_{s=0}^{\ell}
K_{S^1}(\pt) = \bigoplus_{s=0}^{\ell} R(S^1) = \bigoplus_{s=0}^{\ell}
\Z[u,u^{-1}],
\]
where $R(S^1)$ denotes the representation ring of $S^1$ and we use $u$
as the variable in $R(S^1)$. For the first isomorphism above, we use
the fact that the $S^1$-action on $\C^{n+1}$ is linear, so any
fixed-point set for any group element is an $S^1$-invariant affine
subspace of $\C^{n+1}$, hence $S^1$-equivariantly contractible to a
point.

We denote
by $\alpha_s$ the element in $NK^{\Gamma}_{S^1}(\C^{n+1})$ which is
the identity \(1 \in K_{S^1}(\pt) \cong R(S^1)\) in the summand
corresponding to $\zeta_s$ and is $0$ elsewhere. These are clearly
additive $K_{S^1}$-module generators of
$NK^{\Gamma}_{S^1}(\C^{n+1})$. Hence,
in order to determine the
ring structure of the $\Gamma$-subring of the inertial $K$-theory, it
suffices
to calculate the products $\alpha_s \kt \alpha_t$, for \(0 \leq s,t
< \ell,\) in inertial $K$-theory. Since $\C^{n+1}$ with the given
$S^1$-action is a Hamiltonian $S^1$-space with a component of the
moment map which is proper and bounded below (and the fixed point set has finitely many
connected components), from
Section~\ref{sec:definition} we know that the map $\iota_{NK}^*$ is
injective (and in fact, in this case, is an isomorphism).
Hence we may use for
our computations the
$\star$ product on inertial $K$-theory as defined in
Section~\ref{sec:definition}, instead of the $\kt$ product.

For any integer \(s \in \Z,\) let $[s]$ denote the smallest
non-negative integer congruent to $s$ modulo $\ell$. Also let
\(\langle q \rangle := q - \lfloor q \rfloor\) denote the fractional
part of any rational number \(q \in \Q.\)
In our case, the logweight of an element \(\zeta_s \in \Gamma\) acting
on the $k$-th coordinate may be explicitly computed to be
\[
a_k(\zeta_s) = \frac{[b_k s]}{\ell} = \left\langle \frac{b_k s}{\ell} \right\rangle.
\]
Hence $\zeta_s$ acts on the $k$-th coordinate as \(e^{2\pi i
  a_k(\zeta_s)}.\) By the formula for the $\star$ product in Definition~\ref{def:starproduct},
we immediately obtain the relation
\begin{equation}\label{eq:prod-inertial-wps}
\alpha_s \star \alpha_{s'}  = \alpha_{[s+s']} \left( \prod_{k=0}^n
\left(1-u^{-b_k}\right)^{a_k(\zeta_s)+a_k(\zeta_{s'})-a_k(\zeta_s \zeta_{s'})} \right)
\end{equation}
among the generators of the twisted sectors, where we have used that
the $S^1$-equivariant $K$-theoretic Euler class of the
$S^1$-equivariant bundle \(\C_{\lambda} \to \pt\) of $S^1$-weight
$\lambda \in \Z$ is
\[
\epsilon_{S^1}(\C_{\lambda}) = 1 - u^{- \lambda} \in R(S^1) = \Z[u,u^{-1}].
\]
Hence the $\Gamma$-subring may be described as
\begin{equation}\label{eq:inertial-wps}
NK^{\Gamma}_{S^1}(\C^{n+1}) \cong \Z[u,u^{-1}][\alpha_0, \ldots,
  \alpha_{\ell}] \bigg/ {\mathcal I},
\end{equation}
where ${\mathcal I}$ is the ideal generated by the
relations~\eqref{eq:prod-inertial-wps} for all \(0 \leq s,s' \leq \ell-1,\)
i.e.
\begin{equation}\label{eq:relations-inertial-wps}
{\mathcal I} := \left\langle \alpha_s\alpha_{s'} - \left(
\prod_{k=0}^n
\left(1-u^{-b_k}\right)^{a_k(\zeta_s)+a_k(\zeta_{s'})-a_k(\zeta_s
\zeta_{s'})}\right) \alpha_{[s+{s'}]}
\, \bigg\vert \, \, 0 \leq s,s' \leq \ell-1 \right\rangle.
\end{equation}
In order to obtain the orbifold $K$-theory of the symplectic quotient
$\P^n_b$, we must now compute the kernel of the $K$-theory Kirwan map
for each sector, i.e.
\[
\kappa_s: K_{S^1}((\C^{n+1})^{\zeta_s}) \to
K_{S^1}((\Phi^{-1}(-1))^{\zeta_s})
\]
for each \(0 \leq s \leq \ell-1.\) (Here we have chosen to reduce at
the regular value $-1$.) As mentioned above, since each
$(\C^{n+1})^{\zeta_s}$ is itself a Hamiltonian $S^1$-space and its
symplectic quotient by this $S^1$ is again a toric variety, we may
apply \cite[Theorem 3.1]{HL-kernel} to
obtain
\begin{equation}\label{eq:kernel-Kirwan-wps}
\ker(\kappa_s) = \left\langle \prod_{k: a_k(\zeta_s)=0, 0 \leq k \leq n}
(1-u^{-b_k}) \right\rangle.
\end{equation}

Here we use the fact that the $k$-th coordinate line in $\C^{n+1}$ is
fixed by $\zeta_s$ if and only if \(a_k(\zeta_s) = 0,\) and that in this
case (where we have just an $S^1$-action, not a larger-dimensional
torus) the negative normal bundle with respect to $\Phi$ is in fact
all of the tangent space to the fixed point $\{0\}$ in
$(\C^{n+1})^{\zeta_s}$.

Hence we conclude that the full orbifold $K$-theory
of the weighted projective space $\P^n_b$ is given by
\[
\okf(\P^n_b) = \Z[u,u^{-1}, \alpha_0, \ldots, \alpha_{\ell}] \bigg/
{\mathcal I} + {\mathcal J},
\]
where ${\mathcal I}$ is the ideal in~\eqref{eq:relations-inertial-wps}
and
\[
{\mathcal J} = \left\langle \prod_{k:a_k(\zeta_s)=0, 0 \leq k \leq n}
(1-u^{-b_k}) \, \bigg\vert \, 0 \leq s \leq \ell-1 \right\rangle.
\]

Simple algebra then shows that the full orbifold
$K$-theory of a weighted projective spaces is torsion-free.

\begin{proposition}\label{prop:torsionfree}
The full orbifold $K$-theory ring of a weighted projective space $\P^n_b$ 
obtained as a symplectic quotient
does not contain (additive $\Z$) torsion.
\end{proposition}

\begin{proof}
It is sufficient to check that each summand
\[
K^*_{S^1}((\Phi^{-1}(1))^{\zeta_s})
\]
is torsion-free over $\Z$.  This piece is precisely the ring
\begin{equation}\label{eq:wps-piece}
A = \frac{\Z[u,u^{-1}]}{\left \langle \prod_{k: a_k(\zeta_s)=0, 0 \leq k \leq n}
(1-u^{-b_k}) \right\rangle}.
\end{equation}
Now suppose that $f\in A$ is a torsion class: that is, there is an
integer $m\geq 2$ satisfying $m\cdot f = 0$ in $A$.  Let $F\in
\Z[u,u^{-1}]$ be a representative of $f$.  Then $mF$ must be in the
ideal in the denominator of \eqref{eq:wps-piece}.  But $\Z[u,u^{-1}]$
is a unique factorization domain, and $m$ is not a unit in
$\Z[u,u^{-1}]$. Thus, since
\[
mF = \tau \cdot  \prod (1-u^{-b_k}),
\]
unique factorization implies that $\tau$ is multiple of $m$.  Thus we
deduce that $F$ itself is in the ideal, and hence $f = 0$ in $A$.
\end{proof}

\subsection{A worked example: $\P^2_{(1,2,4)}$}

We now illustrate the computations above using the specific weighted projective space
$\P^2_{(1,2,4)}$, the orbifold that is the symplectic quotient of $\C^3$ equipped with the
$S^1$-action
\[
t \cdot (z_0, z_1, z_2) := (tz_0, t^2z_1, t^4z_2).
\]
for \(t \in S^1, (z_0, z_1, z_2) \in \C^3.\)
We will
denote the weight spaces of this $\C^3$ by \(\C_{(1)},
\C_{(2)}, \C_{(4)}\) respectively.
In this case \(\ell = \lcm(1,2,4)=4\) so \(\Gamma \cong \Z_4,\)
generated by \(e^{i2\pi/4} = i\in S^1.\) The
following chart contains the necessary information for computing
the inertial $K$-theory of $\C^3$.

\renewcommand{\arraystretch}{1.3}
\begin{equation}\label{eq:P2-124-chart}
\begin{array}{c||c|c|c|c|}
s & 0 & 1 & 2 & 3 \\ \hline \hline
\zeta_s & 1 & i & -1 & -i \\ \hline
(\C^3)^{\zeta_s} & \C^3 & \C_{(4)} & \C_{(2)} \oplus \C_{(4)} & \C_{(4)} \\
  \hline
a_1(\zeta_s) & 0 & \frac{1}{4} & \frac{1}{2} & \frac{3}{4} \\ \hline
a_2(\zeta_s) & 0 & \frac{1}{2} & 0 & \frac{1}{2} \\ \hline
a_3(\zeta_s) & 0 & 0 & 0 & 0 \\ \hline
\genfrac{}{}{0pt}{0}{\mbox{generator of}}{\zeta_s \, \, \mbox{sector}} &
\alpha_0 & \alpha_1 & \alpha_2 & \alpha_3 \\ \hline
\end{array}
\end{equation}

Using the formula~\eqref{eq:prod-inertial-wps} given above, we
immediately conclude that the product structure in the inertial
$K$-theory of $\C^3$ is given by the following multiplication
table. Recall that $\alpha_0$, being the generator of the untwisted
(identity)
sector, is the multiplicative identity in the ring, so we need not
include it in this table.

\renewcommand{\arraystretch}{1.3}
\begin{equation}\label{eq:P2-124-mult-table}
\begin{array}{c||c|c|c|}
 & \alpha_1 & \alpha_2 & \alpha_3  \\ \hline \hline
\alpha_1  & (1-u^{-2})\alpha_2 & \alpha_3 & (1-u^{-1})(1-u^{-2})\alpha_0  \\ \hline
\alpha_2 &  & (1-u^{-1})\alpha_0 & (1-u^{-1})\alpha_1 \\  \hline
\alpha_3  &  &  & (1-u^{-1})(1-u^{-2})\alpha_2 \\ \hline
\end{array}
\end{equation}

Let ${\mathcal I}$ be the ideal generated by the product relations
in~\eqref{eq:P2-124-mult-table}. Then we have
\[
NK^{\Gamma}_{S^1}(\C^3) \cong \Z[u,u^{-1}][\alpha_0, \alpha_1,
  \alpha_2, \alpha_3] \bigg/ {\mathcal I} + \langle \alpha_0 -1
  \rangle.
\]
We may also explicitly compute the kernels of the $K$-theory Kirwan
maps $\kappa_s$ for \(0 \leq s \leq 3,\) according
to~\eqref{eq:kernel-Kirwan-wps}. We have
\begin{align*}
\ker(\kappa_0) & = \langle \alpha_0 (1-u^{-1})(1-u^{-2})(1-u^{-4}) \rangle, \\
\ker(\kappa_1) & = \langle \alpha_1 (1-u^{-4}) \rangle, \\
\ker(\kappa_2) & = \langle \alpha_2 (1-u^{-2})(1-u^{-4}) \rangle, \\
\ker(\kappa_3) & = \langle \alpha_3 (1-u^{-4}) \rangle. \\
\end{align*}
Let ${\mathcal J}$ be the ideal generated by $\ker(\kappa_s)$ for all
$s$, \(0 \leq s \leq 3.\) We conclude that
\begin{equation}\label{eq:orbi-Kthy-wps}
\okf(\P^2_{1,2,4}) \cong \Z[u,u^{-1}][\alpha_0, \alpha_1,
  \alpha_2, \alpha_3] \bigg/ {\mathcal I} + \langle \alpha_0 -
  1\rangle + {\mathcal J}.
\end{equation}

\def\cprime{$'$}

\end{document}